\documentclass[11pt]{amsart}
\usepackage{latexsym,amssymb,amsmath,youngtab}
\usepackage{amsmath,amsthm,amssymb,amscd}
\usepackage[mathscr]{eucal}
\usepackage{verbatim}
\newtheoremstyle{custom}
  {3pt}
  {3pt}
  {\slshape}
  {}
  {\bfseries}
  {.}
  { }
   {}
\theoremstyle{custom}
\newtheorem{theorem}{Theorem}[subsection]

\newtheorem{proposition}[theorem]{Proposition}
\newtheorem{proposition/definition}[theorem]{Proposition/Definition}
\newtheorem{lemma}[theorem]{Lemma}

\newtheorem{prop}[theorem]{Proposition}

\theoremstyle{definition}

\theoremstyle{remark}

\newtheoremstyle{exercise}
  {3pt}
  {6pt}
  {}
  {}
  {\bfseries}
  {:}
  { }
   {}
\theoremstyle{exercise}
\newtheorem{exercise}[theorem]{Exercise}
\newtheoremstyle{exercises}
  {3pt}
  {6pt}
  {}
  {}
  {\bfseries}
  {:}
  {\newline}
   {}

\theoremstyle{exercise}
\newtheorem{exercises}[theorem]{Exercises}

\def\intprod{\negthinspace
\mathbin{\raisebox{.4ex}{\hbox{\vrule height .5pt width 4pt depth 0pt %
          \vrule height 4pt width .5pt depth 0pt}}}}

\input epsf
\def\boxit#1{\vbox{\hrule height1pt\hbox{\vrule width1pt\kern3pt
  \vbox{\kern3pt#1\kern3pt}\kern3pt\vrule width1pt}\hrule height1pt}}

\def\trank{\text{rank}}

\def\BC{\mathbb C}

\def\BP{\mathbb P}
\def\pp#1{\mathbb P^{#1}}

\def\pp#1{{\mathbb P}^{#1}}
\def\tdim{{\rm dim}}

\def\hd{,...,}

\def\upperp{{}^\perp}

\def\CC{\mathbb C}

\def\11{\mathbf 1}
\def\PP{\mathbb P}

\def\l{\lambda}
\def\a{\alpha}

\def\o{\omega}

\def\b{\beta}

\def\e{\varepsilon}
\def\ot{{\mathord{ \otimes } }}
\def\op{{\mathord{\,\oplus }\,}}

\def\dim{{\rm dim}\;}
\def\La#1{\Lambda^{#1}}

\def\op{\oplus}

\def\ep{\epsilon}
\def\op{\oplus}

\def\t{\tau}

\def\a{\alpha}
\def\b{\beta}

\def\l{\lambda}

\def\ol{\overline}

\def\BP{\mathbb  P}
\def\BC{\mathbb  C}

\def\pp#1{\mathbb  P^{#1}}

\def\ep{\epsilon}

\def\hd{, \hdots ,}

\def\La#1{\Lambda^{#1}}

\def\pp#1{\mathbb  P^{#1}}

\def\tdet{\operatorname{det}}
\def\tperm{\operatorname{perm}}

\def\tdim{\operatorname{dim}}

\def\trank{\operatorname{rank}}

\def\upperp{{}^{\perp}}

\def\be{\begin{equation}}
\def\ene{\end{equation}}

\def\trank{{ {\bold R}}}

\def\trank{{\mathrm {rank}}}

\def\oldet{\ol{GL(W)\cdot [\tdet_n]}} 
\def\oldetc{\ol{GL_{n^2}\cdot [\tdet_n]}}

\newcommand{\tEnd}{\operatorname{End}}

\def\Dual{{\mathcal D}}

\begin{document}

\title[Dual varieties and GCT]{Hypersurfaces with degenerate duals \linebreak
and the Geometric Complexity Theory Program}
\author{J.M. Landsberg, Laurent Manivel and Nicolas Ressayre}
\date{April 2010}
 \begin{abstract} We determine set-theoretic defining equations
 for the variety $Dual_{k,d,N}\subset \BP (S^d\BC^N)$ of hypersurfaces of degree $d$ in $\BC^N$ 
  that have
 dual variety of dimension at most $k$.  We apply these equations
 to the Mulmuley-Sohoni variety $\oldetc\subset \BP (S^n\BC^{n^2})$,
 showing it is an irreducible component of  the variety of hypersurfaces
 of degree $n$ in $\BC^{n^2}$ with dual of dimension at most $2n-2$.
 We    establish additional  geometric properties of the Mulmuley-Sohoni variety and prove a quadratic
 lower bound for the determinental border-complexity of the permanent.  
\end{abstract}
\thanks{ Landsberg  supported by NSF grant  DMS-0805782.\\
NR supported by the French National Research Agency (ANR-09-JCJC-0102-01)}
\email{jml@math.tamu.edu,  laurent.manivel@ujf-grenoble.fr, ressayre@math.univ-montp2.fr}
\maketitle

\section{Introduction}

A  classical problem in linear algebra is to  determine    or bound the smallest integer $n$ such that the permanent 
of an $m\times m$ matrix  may be realized as a linear projection of the determinant
of an $n\times n$ matrix.  L. Valiant \cite{MR526203} proposed 
using this problem as an algebraic analog of the problem of comparing the complexity classes
$\bold{P}$ and $\bold{NP}$.
Call this value of $n$, $dc(\tperm_m)$.  He conjectured that $dc(\tperm_m)$
grows faster than any polynomial in $m$.
The best known lower bound is 
$dc(\tperm_m)\geq \frac {m^2}2$, which was proved
in \cite{MR2126826}.

The definition of $dc(\tperm_m)$ may be rephrased as
follows: let $\ell$ be a linear coordinate on $\BC$, let $\BC\op M_m(\CC)\subset M_n(\CC)$,
be any linear inclusion, where $M_n(\CC)$ denotes the space of complex $n\times n$ matrices;
then $dc(\tperm_m)$ is the smallest
$n$ such that   
$\ell^{n-m}\tperm_m \in \tEnd(M_n(\CC))\cdot \tdet_n$. Here
$u\in \tEnd(M_n(\CC))$ acts by  $(u\cdot \tdet_n)(M):=\tdet_n(u(M))$.

K. Mulmuley and M. Sohoni \cite{MS1,MS2}, have proposed to study the
function $\ol{dc}(\tperm_m)$, which is the smallest 
$n$ such that $[\ell^{n-m}\tperm_m]$ is contained in 
the orbit closure $\oldetc\subset \BP(S^n(M_n(\CC))^*)$.
Here $S^n(M_n(\CC))^*$ denotes the homogeneous polynomials of
degree $n$ on $M_n(\BC)$.
The best known lower bound on this function had been linear.
Note that $\ol{dc}(\tperm_m)\leq dc(\tperm_m)$, the potential
difference being the added flexibility of limiting
polynomials in $\oldetc$ that are not in 
$End(M_n(\CC))\cdot [\tdet_n]$. 
Our main result about $\ol{dc}(\tperm_m)$ is the following quadratic bound.\\

\begin{theorem}
  $\ol{dc}(\tperm_m)\geq \frac{m^2}2$.
\end{theorem}

\bigskip
Consider the ideal   of regular functions on 
$S^n(M_n(\CC)^*) $ that are zero on $\ol{GL_{n^2}\cdot [\tdet_n]}$.
We construct an explicit sub-$GL_{n^2}$-module $V_n$ in this ideal which has the following
properties.\\

\begin{theorem}\label{th:eqdet}
  \begin{enumerate}
  \item The $GL_{n^2}$-module $V_n$ contains an irreducible module of highest weight 
$$n(n-1)(n-2)\o_1+ (2n^2-4n-1)\o_2 + 2\o_{2n+1}$$
 and $V_n$ is a subspace of the space of homogeneous polynomials of   degree $n(n-1)$ on $S^n(M_n(\CC))^*$.
\item The variety $\ol{GL_{n^2}\cdot [\tdet_n]}$ is an irreducible component of the zero locus
$\Dual_n$ of $V_n$.
  \end{enumerate}
\end{theorem}

\bigskip
Theorem \ref{th:eqdet} provides the first explicit module
of equations in the ideal of $\oldetc$. However $\tdim (\Dual_n)$ grows
exponentially with $n$,   whereas $\tdim(\oldetc)$ is on the order of $n^4$.
In particular,  $\Dual_n$ has other irreducible components, one of which is described in \S 4.
A more precise statement than Theorem~\ref{th:eqdet} is  Theorem \ref{smooth}, which  
implies that  our equations provide a full set 
of local equations of $ \oldetc$ around $[\det_n]$.


\medskip

One can similarly define $dc(P),\ol{dc}(P)$ for an arbitrary polynomial $P$ of degree $n$ in $N$
variables. Such a polynomial, if nonzero, defines a hypersurface $Z(P)\subset\PP^{N-1}$. If $P$ is
reduced, the 
Zariski closure of the set of tangent hyperplanes to this hypersurface is a subvariety $Z(P)^*$
of the dual projective space, called the {\it dual variety} of $Z(P)$.  For generic such $P$,
$Z(P)^*$ is a hypersurface.

\bigskip
\begin{theorem} For any  irreducible   polynomial $P$,  
\label{oldcbnd} $$\ol{dc}(P)\geq \frac{\dim Z(P)^*+1}2.$$
\end{theorem}

Theorem  \ref{oldcbnd} is obtained by partially solving    
a question in classical algebraic geometry (Theorem \ref{degdualeqns}): find
set-theoretic defining equations for the variety $Dual_{k,d,N}\subset \BP(S^d\BC^N)$ of hypersurfaces of
degree $d$ in $\BC\pp{N-1}$ whose dual variety has dimension at most $k$.

\bigskip
While it was generally understood that $\tEnd(M_n(\CC))\cdot [\tdet_n]
\subset \oldetc$ was a proper inclusion, it had not been known if the difference
was potentially significant. Proposition \ref{Lambda} exhibits an explicit
codimension one $GL_{n^2}(\CC)$-orbit that is contained in  the boundary of $\oldetc$ but not
contained in $\tEnd(\BC^{n^2})\cdot \tdet_n$, at least when $n$ is odd.
In particular, we exhibit   an explicit sequence of polynomials $P_m$ with
$\ol{dc}(P_m)<dc(P_m)$.\\

\section{Hypersurfaces with degenerate dual varieties}

\subsection{Katz's dimension formula}\label{katzformula}

Let $W$ be a complex vector space of dimension $N$, and $P\in S^dW^*$
a homogeneous polynomial of degree $d$. Let $Z(P)\subset\PP W$ denote the 
hypersurface defined by $P$. If $P$ is irreducible, 
then $Z(P)$ and its dual variety $Z(P)^*$, the Zariski closure of the 
set of tangent hyperplanes to $Z(P)$, are both irreducible. 
The Katz dimension formula \cite{GKZ} states that 
$$\dim Z(P)^* = \trank (H_{P,w})-2,$$
where $H_{P,w}$ denotes the Hessian of $P$ at $w$, a general point of the affine 
cone over $Z(P)$. 
Recall that the Hessian can be defined, once a coordinate system on $W$ has been
chosen, as the symmetric matrix of second partial derivatives of $P$. Intrinsically, 
it is just the quadratic form constructed from $P$ by polarization: 
$$H_{P,w}(X):= P(w,\ldots ,w,X,X).$$

Katz's formula implies that $Z(P)^*$ has dimension less
or equal to $k$ if and only if,
for any $w\in W$ such that $P(w)=0$, and any $(k+3)$-dimensional subspace $F$ of $W$, 
$$\det (H_{P,w}|_F) =0. $$
  
Equivalently (assuming $P$ is  irreducible), for any such subspace $F$, the polynomial $P$ must divide   $\det (H_P|_F)$,
a polynomial of degree $(k+3)(d-2)$.

\subsection{Pairs of polynomials such that one divides the other}

Consider two homogeneous polynomials $P,Q$ on $W=\BC^N$, of respective degrees $d,e$. 
We determine    equations on their coefficients that are implied by the 
condition that $P$ divides $Q$. 

There is an obvious solution in the slightly different situation where $P$ and $Q$
are (non-homogeneous) polynomials in a single variable: one simply   performs 
the Euclidian division of $Q$ by $P$ and requires that the remainder $R$ be zero. The 
ideal defined by this condition is described in \cite{MR898165}. 

In our   situation, we can first restrict $P$ and $Q$ to some plane $L$ in
$W$, and choose coordinates $x,y$ on $L$. The restricted polynomials $P_L$ and $Q_L$ 
are then binary forms in these coordinates. Then     set  $y=1$ and perform
a  Euclidean division on the resulting polynomials in $x$. After rehomogenization, 
we obtain
\begin{equation}\label{euclide}
Q_L(x,y) = P_L(x,y)D_L(x,y)+y^{e-d+1}R_L(x,y),
\end{equation}
where $R_L(x,y)$ is homogeneous of degree $d-1$. The condition $R_L=0$
depends on the choice of the coordinates $x$ and $y$, but up to scale,    the 
coefficient $R_{L,d-1}$ of $x^{d-1}$   only depends on the choice of the
coordinate $y$. That is,  the condition $R_{L,d-1}=0$,
considered as a polynomial equation in the coefficients of $P$ and $Q$, 
  only depends on the choice of $L$ and of the line $D$ in $L$ defined by 
the equation $y=0$. 

To make the connection with \cite{MR898165},   write 
$$\begin{array}{l}
Q_L(x,y)=\sum_{i=0}^eq_ix^iy^{e-i}=q_e\prod_{k=1}^e(x-y\alpha_k), \\
P_L(x,y)=\sum_{j=0}^dp_jx^jy^{d-j}=p_d\prod_{l=1}^d(x-y\beta_l).
\end{array}$$
Divide equation (\ref{euclide}) by $P_L(x,y)$ and set $x=1$. We get an identity between
power series in $y$, to which $D_L$ contributes only up to degree $e-d$. We conclude that 
$R_{L,d-1}/p_d$ is  equal to the coefficient of $y^{e-d+1}$ in 
$$\frac{Q_L(1,y)}{P_L(1,y)}=\frac{q_e\prod_{k=1}^e(1-y\alpha_k)}{p_d\prod_{l=1}^d(1-y\beta_l)}
=\frac{q_e}{p_d}\sum_{m\ge 0}s_m(\beta-\alpha)y^m,$$
where the last equality can be taken as a definition of the symmetric functions $s_m(\beta-\alpha)$. 
The condition that $R_{L,d-1}=0$ is thus equivalent to the condition that 
$$s_{e-d+1}(\beta-\alpha)=0.$$ 
In order to get a polynomial equation in the coefficients of $Q_L$ and $P_L$, we    modify
the expression slightly. Write  
$$\frac{Q_L(1,y)}{P_L(1,y)}=\frac{Q_L(1,y)}{p_d(1+\pi(y))}=\frac{Q_L(1,y)}{p_d}\sum_{m\ge 0}(-1)^m\pi(y)^m,$$
where $\pi(y)=\sum_{j=1}^d\frac{p_{d-j}}{p_d}y^j$. Therefore, the coefficient of $y^{e-d+1}$ can be 
expressed as
$${ 
\hat R(Q,P):=}\frac 1 {p_d}\sum_{i=0}^eq_i\sum_{j_1+\cdots +j_r=-d+1+i}(-1)^r\frac{p_{d-j_1}\cdots p_{d-j_r}}{p_d^r}.$$
In that sum the maximal value of $r$ is $e-d+1$, so we make it a polynomial  by multiplying
by $p_d^{e-d+2}$. We conclude that $R_{L,d-1}=0$ is equivalent to the condition that 
\begin{equation}\label{eqdual}
\sum_{j_1+\cdots +j_r=-d+1+i}(-1)^rq_i p_{d-j_1}\cdots p_{d-j_r}p_d^{e-d+1-r}=0.
\end{equation}
This condition is linear in the coefficients of $Q_L$, and of degree $e-d+1$ in those of $P_L$. 
It depends, as we have seen, on the choice of a preferred coordinate on $L$, in particular, 
on the choice of the line $D$ defined by this coordinate. 

Note the following behavior under rescaling:
  \begin{eqnarray}
    \label{eq:R}  
\hat R(\alpha Q(x,\lambda y),\beta P(x,\lambda y))=
\alpha\beta^{e-d+1}\lambda^{e-d+1}\hat R(Q,P).
\end{eqnarray}

\subsection{Equations for hypersurfaces with degenerate duals} 
We   apply the results of the preceeding section to the case where $Q=\det (H_{P}|_F)$,
whose degree equals $e=(k+3)(d-2)$. Recall that $F\subset W$ is a subspace of dimension $k+3$. 
Once $F$ has been chosen, we obtain a family of equations depending, up to scale, only on the 
choice of a plane $L$ in $W$ and a line $D$ in $L$. In particular, if $F$ contains $L$ 
we get an equation depending only on the partial flag $D\subset L\subset F$. This equation
must therefore be a highest weight vector in some module of polynomials on $S^nW^*$, and
its highest weight must be of the form $a\omega_1+b\omega_2+c\omega_{k+3}$.

   Consider a basis adapted to $D\subset L\subset F$ and let 
$(x,y,z,w)=(x,y,z^i,w^s)$   denote its dual basis.
Consider a diagonal matrix $T:=(t_x,t_y,t_zId_{F/L},t_wId_{W/F})$.
Under rescaling
\begin{eqnarray}
  \label{eq:P}
(T.P)(x,y,0,0)=t_x^{-n}P(x,t_xt_y^{-1}y,0,0).  
\end{eqnarray}

Moreover, the matrix of ${H_{T.P}}|_F$ is obtained from that of ${H_P}|_F$ by substituting
$(x,y,z,w)$ by \linebreak $(t_x^{-1}x,t_y^{-1}y,t_z^{-1}z,t_w^{-1}w)$ and multiplying the first row and
column by $t_x^{-1}$, the second row  and column by $t_y^{-1}$ and the other rows and columns
by $t_z^{-1}$. It follows that   $\tdet({H_{T.P}}|_F)$  is obtained from $\tdet({H_{P}}|_F)$ by substituting 
 $(t_x^{-1}x,t_y^{-1}y,t_z^{-1}z,t_w^{-1}w)$ in for $(x,y,z,w)$   and by multiplying the result by
$$
t_x^{-2}t_y^{-2}t_z^{-2(k+1)}.
$$
In summary, 
\begin{eqnarray}
  \label{eq:Q}
  \tdet({H_{T.P}}|_{F})(x,y,0,0)=t_x^{-2}t_y^{-2}t_z^{-2(k+1)}t_x^{-e}\tdet({H_{ P}}|_{F})(x,t_xt_y^{-1}y,0,0).
\end{eqnarray}
From equations \eqref{eq:R}, \eqref{eq:P} and \eqref{eq:Q},  the vector of exponents of the action of $T$ on our equation is:
$$
 \left(
  \begin{array}{c}
2+e+(d-1)(e-d+1)\\e-d+3\\2(k+1) 
  \end{array}
\right).
$$

This vector should be 
$$\left(
  \begin{array}{c}
a+b+c\\b+c\\c(k+1) 
  \end{array}
\right).
$$
We deduce  
$$
\begin{array}{l}
  a = -e + 3 d - 2 + d e - d^2 =(d-1)(d-2)(k+2), \\ 
b = e - d+1=d(k+2)-2k-5,\\ 
c = 2.
\end{array}
$$
Note that $a+2b+(k+3)c= d(d-1)(k+2)$ so this module occurs in $W^{\ot   d(d-1)(k+2)}$.

Define $Dual_{k,d,N}\subset \PP(S^dW^*)$ as the Zariski closure of the set of 
irreducible hypersurfaces of degree $d$ in $\PP W\simeq \PP^{N-1}$, whose dual variety has dimension at most $k$. 

\bigskip 
\begin{theorem}\label{degdualeqns}
The variety $Dual_{k,d,N}\subset \PP(S^n(\BC^N)^*)$ has equations given by a copy of the $SL_N$-module 
  with highest weight 
$$\Omega(k,d)= (d-1)(d-2)(k+2) \omega_1+\big( d(k+2)-2k-5\big) \omega_2+2\omega_{k+3}.$$
These equations have degree $(k+2)(d-1)$.
\end{theorem}

\bigskip
Note that when we constructed our equations, we did not suppose that $L$ was contained in $F$. 
This indicates that the module generated by these equations should in fact 
be larger than the single module with highest weight $\Omega(k,d)$.

Set theoretically, these equations suffice to define $Dual_{k,d,N}$ locally, at least on the open 
subset parametrizing irreducible hypersurfaces $Z(P)\subset \PP(W)$. Indeed, once the plane 
$L$ is fixed, by varying the line $D$ one obtains a family of equations expressing
the condition that $P_L$ divides $Q_L$, respectively the restrictions to $L$ of the polynomials $P$ and 
$Q=\det (H_{P}|_F)$. But $P$ divides $Q$ if and only if    restricted to each
plane $P$ divides $Q$, so our conditions   imply that the dual variety of the irreducible hypersurface $Z(P)$ has dimension less or equal
to $k$. On the other hand, if $P$ is not reduced, then these equations can vanish even if the
dual of $P_{red}$ is non-degenerate. For example, if $P=R^2$ where
$R$ is a quadratic polynomial of rank $2s$, then $\tdet(H_P)$ is a multiple of $R^{2s}$.

\subsection{Polynomials of the form $\ell^{d-m}R$}

\begin{lemma}\label{linfaclem} Let $U=\BC^M$ and $L=\BC$, let  $R\in S^m U^*$ be  irreducible,
let $\ell\in L^*$ be nonzero,  let $U^*\op L^*\subset W^*$ be a linear inclusion,
and let $P=\ell^{d-m}R\in S^dW^*$. 

If $[R]\in Dual_{\kappa, m, M}$ and 
$[R]\not\in Dual_{\kappa-1, m, M}$, then $[P]\in Dual_{\kappa, d, N}$ and 
$[P]\not\in Dual_{\kappa-1, d, N}$.
\end{lemma}

\begin{proof}
Choose a basis $u_1\hd u_M,v,w_{M+2}\hd w_N$ of $W$ so $(U^*)\upperp =\langle w_{M+2}\hd w_N\rangle$
and $(L^*)\upperp =\langle u_1\hd u_M, w_{M+2}\hd w_N\rangle$. Let $c=(d-m)(d-m-1)$. In these coordinates, we
have the matrix in $(M,1,N-M-1)\times (M,1,N-M-1)$-block form:
$$
H_P=
\begin{pmatrix} \ell^{d-m}H_R &0&0\\
0& c\ell^{d-m-2}R & 0\\
0&0&0
\end{pmatrix}
$$

First note that $\tdet_{M+1}(H_P|_F)$ for any $F=\BC^{M+1}$ is either zero or a multiple of $P$. If
$\tdim Z(R)^*=M-2$ (the expected dimension), then 
 for a generic  $F=\BC^{M+1}$, $\tdet_{M}(H_P|_F)$ will not be a multiple of $P$, and more
generally if 
 $\tdim Z(R)^*=\kappa$, then for a generic 
$F=\BC^{\kappa + 2}$, $\tdet_{\kappa + 2}(H_P|_F)$ will not be a multiple of $P$ but 
for any 
$F=\BC^{\kappa + 3}$, $\tdet_{\kappa + 3}(H_P|_F)$ will   be a multiple of $P$.
\end{proof}

\section{The orbit of the determinant}

\subsection{Statement of the main result}

Let $W=M_n(\CC)$, the space of complex matrices of size $n$. Its dimension is $N=n^2$. 
The hypersurface in $\PP W$ defined
by the determinant is dual to the variety of rank one matrices,   the {\it Segre
product} $\PP^{n-1}\times \PP^{n-1}\subset \PP^{N-1}$. 

Intuitively, a deformation of the determinant 
hypersurface, subject to the condition that its dual remains of dimension $2n-2$,  
should have  a deformation of the Segre as its dual variety. But the Segre is rigid,  its
only deformations in $\PP W^*$ are translates by projective automorphisms. Hence
the only deformations of the determinant hypersurface, with duals of the same 
dimension, should be translates by projective automorphisms as well. 

The problem with this intuitive argument is that the dual map can be highly 
discontinuous, especially in the presence of singularities, and the 
determinant hypersurface is very singular. Nevertheless, the conclusion turns 
out to be correct: 

{\it Every small deformation of  the determinant hypersurface, with dual variety of 
the same dimension, is a translate by a projective automorphism.
}

We will prove a more precise statement. For a polynomial $P$ of degree $n$ on $W$, 
and a $k$-dimensional subspace $F$ of $W$, we have expressed the condition that 
$P$ divides $\det (H_P|_F)$  
in terms of  polynomial equations of degree $(k+2)(n-1)$. 
These equations define a subscheme ${\mathcal Dual}_{k,n,N}\subset \BP S^nW^*$,
supported on the variety $Dual_{k,n,N}$  and possibly  other reducible hypersurfaces.

\bigskip 
\begin{theorem}\label{smooth}
The scheme ${\mathcal Dual}_{2n-2,n,n^2}$ is smooth at $[\det_n]$, and the 
$PGL_{n^2}$-orbit closure of $[\det_n]$ is an irreducible component of ${\mathcal Dual}_{2n-2,n,n^2}$. 
\end{theorem}

\bigskip 
In particular, Theorem \ref{smooth} implies that the $SL(W)$-module of highest weight $ \Omega(2n-2,n) $ given by  ~\eqref{eqdual}
gives  local equations at $[\det_n]$
of $\ol{GL_{n^2}\cdot [\tdet_n]}$, of degree $2n(n-1)$. Since  $Dual_{k,n,N}$ always contains 
the variety of degree $n$ hypersurfaces which are cones over a linear space of dimension 
$N-k-1$, the zero set of the equations is strictly larger than $\ol{GL_{n^2}\cdot [\tdet_n]}$. The so-called {\it subspace variety}
of cones has dimension $\binom{k+n+1}n+(k+2)(N-k-2)-1$. 
For $k=2n-2, N=n^2$, this dimension is bigger than the dimension of the orbit
of $[\det_n]$, and therefore $Dual_{2n-2,n,n^2}$ is not   irreducible. 
We have not yet  been able to find equations that separate the orbit
of $[\det_n]$ from the other components of $Dual_{2n-2,n,n^2}$.

\subsection{Consequences regarding Kronecker coefficients}
A copy of the module with highest weight $n(n-1)(n-2)\o_1+ (2n^2-4n-1)\o_2 + 2\o_{2n+1}$
in $S^{2n(n-1)}(S^n\BC^{n^2})$ is in the ideal of $\oldet$. 

The program suggested in \cite{MS2}   was to separate the determinant and permanent by finding $SL(W)$-modules
in the ideal of $\ol{GL_{n^2}\cdot [\tdet_n]}$ such that their entire isotypic component
was in the ideal. (Also see \cite{BLMW} for explicit statements
regarding Kronecker coefficients needed to carry out the program.) 
This does not occur for the module with highest weight $n(n-1)(n-2)\o_1+ (2n^2-4n-1)\o_2 + 2\o_{2n+1}$.

For example, when  $n=3$,  the module with highest weight $12\o_1 + 5\o_2+2\o_7$ occurs
with multiplicity six in $S^{12}(S^3\BC^9)$, but only one copy of it is in the ideal.

\subsection{Computing the Zariski tangent space}\label{sec32}

We   differentiate the condition that $P$ divides $\det (H_P|_F)$ for each $F$. 
That is,   write   $\det (H_P|_F)=PQ_F$ for some polynomial $Q_F$, and 
consider a curve $P_{\ep}=P+\ep \pi +\ep^2 \t + O(\ep^3)$, inducing a curve $Q_{F,\ep}=Q_F+\ep Q_F'+O(\ep^2)$.
  Up to $O(\ep^2)$,  $H_P$ becomes $H_P+\epsilon H_\pi$ and we deduce the identity
\begin{equation}\label{Zar}
\det(H_P,\ldots ,H_P,H_\pi)|_F=\pi Q_F+P Q_F'.
\end{equation}

To exploit \eqref{Zar}, let $[w]$ be a general 
point of the hypersurface $Z(P)$, so   the rank of the quadratic form $H_{P,w}$ 
 is exactly $k+2$. Let $X$ belong to the kernel of $H_{P,w}$. Let $F'$ be a 
$(k+2)$-dimensional subspace of $W$, transverse to the kernel of $H_{P,w}$, and 
let $F=F'\oplus {\BC} X$. Now compute $\det(H_P,\ldots ,H_P,H_\pi)|_F$
at $w$. In terms of bases adapted to the flag $F'\subset F\subset W$, the matrix of $H_{P,w}$ has zeros in its last row 
and column, since they correspond to $X$, which belongs to the kernel. Removing this
row and column yields an invertible matrix,   corresponding to $H_{P,w}|_{F'}$, as 
  $F'$ is transverse to the kernel. 

Now, $\det(H_P,\ldots ,H_P,H_\pi)|_F$ evaluated at $w$ is the sum of the $k+3$ determinants
obtained by considering the matrix of $H_{P,w}|_F$ and replacing one column by the
corresponding column of $H_{\pi,w}|_F$. If this column is not the last one,
this determinant remains with a zero column, hence equals zero. In case the replaced
column is the last one, since the last row of the matrix of $H_{P,w}|_F$ vanishes, 
the resulting determinant is equal to the determinant of the upper-left block, 
$\det (H_{P,w}|_{F'})$, multiplied by the lower-right entry of $H_{\pi,w}|_F$, that is,
  $H_{\pi,w}(X)=\pi(w\hd w, X,X)$.
 Equation (\ref{Zar}) becomes 
\be\label{Zarb}
\det (H_{P,w}|_{F'}) H_{\pi,w}(X)=\pi(w)Q_F(w).
\ene
Note that $Q_F(w)$   depends on both $w$ and $X$ (since $F$ depends on $X$), but
$\det (H_{P,w}|_{F'})$   only depends on $w$. 

Now specialize \eqref{Zarb}  to the case   $P=\det_n$. Then 
$w$ must be a matrix of rank exactly $n-1$. Write $W=E\ot E^*$, and
as such, it is naturally self-dual via the involution $e\ot \phi\mapsto \phi\ot e$. For
$w\in W$, write $w^*\in E^*\ot E=W^*$ for the image of $w$ under the involution.

\begin{lemma}\label{ker}
Let $w$   be a matrix of rank exactly $n-1$. Then the singular locus of the quadratic form $H_{\det_n,w}$,
$(H_{\det_n,w})_{sing}$
is the space   of $n\times n$  matrices $X$ such that:
$$ 1)\; Im(X)\subset Im(w), \quad 2)\; Ker(X)\supset Ker(w), \quad 3)\; w^*(X)=0.$$
\end{lemma}

\proof Write $w=\phi_1\otimes e_1+\cdots +\phi_{n-1}\otimes e_{n-1}$, for some 
collection $e_1,\ldots, e_{n-1}$ of independent vectors in $E=\CC^n$, and some collection 
$\phi_1, \ldots ,\phi_{n-1}$ of independent linear forms. We complete these collections
into bases by adding a vector $e_n$ and a linear form $\phi_n$. Consider an  
endomorphism $X=\sum_{1\leq i, j\leq n}x_{ij}\phi_i\otimes e_j$.
An easy computation yields
$$H_{\det_n,w}(X)=\det(w,\ldots ,w,X,X)=\sum_{i=1}^{n-1}(x_{nn}x_{ii}-x_{ni}x_{in}).$$
This implies  that the singular locus of the quadratic form $H_{\det_n,w}$ is defined by the conditions 
$x_{ni}=x_{in}=0$ for $1\le i\le n$, and $\sum_{i=1}^{n-1}x_{ii}=0$. The first 
identities are equivalent to the conditions   $Im(X)\subset Im(w)$ and $Ker(X)\supset Ker(w)$. 
The third one is the condition $w^*(X)=0$. 
\qed

\medskip
We   summarize our analysis:

\begin{lemma}\label{zar}
Suppose that   $\pi\in S^nW^*$ belongs to the affine Zariski tangent space
$\hat{T}_{[det_n]}{\mathcal Dual}_{2n-2,n,n^2}$. Then for any matrix $w$ of rank $n-1$,
and any $X\in (H_{\det_n,w})_{sing}$,  
$$H_{\pi,w}(X)=c_{X,w}\pi(w),$$
for some scalar $c_{X,w}$ that does not depend on $\pi$. 
\end{lemma}

\subsection{Immanants}
Recall that each partition $\lambda$ of $n$ defines an irreducible representation 
$[\lambda]$ of the symmetric group ${\mathfrak S}_n$, hence a character $\chi_{\lambda}$. 
The immanant $IM_\lambda$ is the degree $n$ polynomial on $M_n$ defined by the formula
$$IM_\lambda(X)=\sum_{\sigma\in {\mathfrak S}_n}\chi_{\lambda}(\sigma)
x_{1\sigma(1)}\cdots x_{n\sigma(n)}.$$
For example, $[n]$ is the trivial representation and $IM_{(n)}$ is the permanent; $[1^n]$ 
is the sign representation and $IM_{(1^n)}$ is the determinant. 

  Write $M_n(\BC)=A^*\otimes B$ for two copies $A,B$ of $\CC^n$. Since $[\det_n]$ is preserved
by the action of $GL(A)\times GL(B)$ by left-right multiplication, this is also the case
of the Zariski tangent space at $[\det_n]$ of the $GL_{n^2}$-invariant scheme
${\mathcal Dual}_{2n-2,n,n^2}$. But as a $GL(A)\times GL(B)$-module, 
$$S^n(A^*\ot B)^* = \bigoplus_{\lambda}S_\lambda A\otimes S_\lambda B^*,$$
where the sum is over all partitions of $n$. Since this decomposition is multiplicity
free, the submodule $\hat{T}_{[det_n]}{\mathcal Dual}_{2n-2,n,n^2}$ must be the direct 
sum of {\it some} of the components:
$$\hat{T}_{[det_n]}{\mathcal Dual}_{2n-2,n,n^2}
= \bigoplus_{\lambda \in P_n}S_\lambda A\otimes S_\lambda B^*,$$
for some set of partitions $P_n$ to be determined. 
Note that $IM_\lambda$ is contained in the component  
$S_\lambda A\otimes S_\lambda B^*$. Therefore  $\lambda$ belongs to $P_n$ if and
only if $IM_\lambda$ belongs to $\hat{T}_{[det_n]}{\mathcal Dual}_{2n-2,n,n^2}$. 

We   apply Lemma \ref{zar} as follows. Start with a matrix $w$ of rank $n-1$, which we
write as $\sum_{i=1}^ne_i^*\otimes c_i$. There is a dependence relation between $c_1,\ldots ,c_n$,
which we can suppose to be of the form $c_n=\sum_{i=1}^{n-1}\mu_ic_i$. Then 
$w= \sum_{i=1}^{n-1}(e_i^*+\mu_ie_n^*)\otimes c_i$. By Lemma \ref{ker}, $(H_{\det_n,w})_{sing}$ can then
be described as the set of all
$$X=\sum_{i=1}^{n-1}(e_i^*+\mu_ie_n^*)\otimes (\sum_{j=1}^{n-1}\zeta_i^jc_j),$$
where $\sum_{i=1}^{n-1}\zeta_i^i=0$. In bases,   the first $n-1$ columns
$c'_1,\ldots ,c'_{n-1}$ of $X$ are linear combinations of the columns of $w$, and  
$c'_n=\sum_{i=1}^{n-1}\mu_ic'_i$ is then given by the same linear combination as for the last column of $w$. 
We can thus write the entries of $X$ as
$$x_i^k=\sum_{j=1}^{n-1}\zeta_i^jw_j^k, \; i<n, \qquad x_n^k=\sum_{i=1}^{n-1}\mu_ix_i^k.$$
Substituting these expressions into  $H_{IM_\lambda,w}(X)=IM_\lambda(w,\ldots ,w,X,X)$ yields
 a polynomial $IM_\lambda(\zeta, w',\mu)$ which is quadratic in the $\zeta_i^j$ and of degree $n$
in the coefficients $w_j^k$, $j<n$, of the first $n-1$ columns of $w$, denoted by $w'$. 
Explicitly, 
$$\begin{array}{rcl}
H_{IM_\lambda,w}(X) &=& IM_\lambda(\zeta, w',\mu)\\
  & = & \sum_{i<j}\sum_{p,q}\zeta_i^p\zeta_j^q\Big(\sum_{k,\sigma}\mu_k
\chi_\lambda (\sigma)w_1^{\sigma(1)}\cdots w_p^{\sigma(i)}\cdots w_q^{\sigma(j)}\cdots 
w_{n-1}^{\sigma(n-1)}  w_k^{\sigma(n)}\Big) \\
 & & +\sum_{i,j}\sum_{p,q}\zeta_i^p\zeta_j^q
\Big(\sum_{\sigma}\mu_j
\chi_\lambda (\sigma)w_1^{\sigma(1)}\cdots w_p^{\sigma(i)}\cdots w_j^{\sigma(j)}\cdots 
w_{n-1}^{\sigma(n-1)}  w_q^{\sigma(n)}\Big).
\end{array}$$
On the other hand,   expressing the last column of $w$ in terms of the first ones, $IM_\lambda(w)$
becomes a polynomial $IM_\lambda(w',\mu)$, of degree $n$ in $w'$:
$$IM_{\l}(w)= IM_\lambda(w',\mu)=\sum_{k,\sigma}\mu_k\chi_\lambda (\sigma)w_1^{\sigma(1)}\cdots w_k^{\sigma(k)}
\cdots w_{n-1}^{\sigma(n-1)}  w_k^{\sigma(n)}.$$
 By Lemma \ref{zar},
  for each choice of $\mu$, the vanishing of $IM_\lambda(w',\mu)$ implies the 
vanishing of $IM_\lambda(\zeta, w',\mu)$. Since they are both homogeneous of degree $m$ in $w'$,
they must be proportional. 

This gives many relations, one for each quadratic monomial in the $\zeta$'s (but recall the 
relation $\sum_{i=1}^{n-1}\zeta_i^i=0$). We will need only a small subset of them:

\begin{lemma}\label{char}
Suppose that $IM_\lambda$ belongs to $\hat{T}_{[det_n]}{\mathcal Dual}_{2n-2,n,n^2}$. 
Then for any permutation $\sigma$, and any triple of distinct integers $i,p,q$ smaller than $n$,
one has the relations
$$\sum_{\tau\in \langle (ip),(qn)\rangle}\chi_\lambda(\sigma\tau)=0.$$
\end{lemma}

Here   $\langle (ip),(qn)\rangle$ denotes the group of permutations generated by the two simple 
transpositions $(ip)$ and $(qn)$. This group has order four, hence we get a collection of 
four term relations between the values of the character $\chi_\lambda$. Observe also 
that since the characters are class functions,   $ipqn$ can be replaced by any four-tuple of
distinct integers. 

\proof Consider the coefficient of $\zeta_i^p\zeta_i^q$ in  $IM_\lambda(\zeta, w',\mu)$. 
It is  
$$\sum_{\sigma}\mu_i
\chi_\lambda (\sigma)w_1^{\sigma(1)}\cdots w_p^{\sigma(i)}\cdots w_p^{\sigma(p)}
\cdots w_i^{\sigma(i)}\cdots  w_q^{\sigma(q)}\cdots  w_q^{\sigma(n)}.$$
The monomials in that sum do not appear in  $IM_\lambda(w',\mu)$, so this sum must be zero. 
Our condition is then just that the coefficient of each monomial is equal to zero, since the 
monomial to which a permutation $\sigma$ contributes does not change when we compose it on the right
with some element of $\langle (ip),(qn)\rangle$. \qed 

\medskip
We conclude:\medskip

\begin{prop} $P_n=\{1^n, 21^{n-2}\}.$ \end{prop}

\proof We know that both partitions are contained in $P_n$, since the first one corresponds
to the determinant itself, and the second one to the tangent space to the orbit of its
projectivization. Therefore, by Lemma \ref{char}, 
it is enough to check that the vector space $C_n$ of class functions
$F$ on ${\mathfrak S}_n$, such that 
$$\sum_{\tau\in (ij)(kl)}F(\sigma\tau)=0 \qquad \forall \sigma,\; \forall i,j,k,l,$$
is at most two-dimensional. We prove that $F\in C_n$ is completely determined by 
its values on permutations of cycle type $(1^n)$ or $(21^{n-2})$.  Recall that the value
of a class function $F$ on a permutation $\sigma$ only depends on its cycle type, 
which is encoded by a permutation $\lambda$. We will thus write $F(\lambda)$ 
rather than $F(\sigma)$. 
Apply induction on the number of fixed points in $\sigma$. Suppose that $\sigma$ has 
at least two nontrivial cycles. Choose $i$ and $k$ in these two cycles and let $j=\sigma(i)$,
$l=\sigma(k)$, then the three permutations $\sigma (ij)$, $\sigma (kl)$,  $\sigma (ij)(kl)$
have more fixed points than $\sigma$. If $\sigma$ has a cycle of length at least four, 
  take $i$ in this cycle and let $j=\sigma(i)$, $k=\sigma(j)$, $l=\sigma(k)$, to obtain the same
conclusion. Finally, if $\sigma$ is of cycle type $31^{n-3}$, say with a nontrivial cycle
$(123)$,  choose $ijkl=1234$. This gives the relation $2F(31^{n-3})+F(41^{n-4})+F(21^{n-2})=0$.
On the other hand, when $\sigma$ has cycle type $41^{n-4}$, with nontrivial cycle $(1234)$, 
  let $ijkl=1324$, which yields the relation $F(41^{n-4})+F(221^{n-4})=0$. And if 
$\sigma$ has cycle type $221^{n-4}$, with nontrivial cycles $(12)(34)$, 
letting $ijkl=1234$ gives the relation $F(221^{n-4})+2F(21^{n-2})+F(1^n)=0$.
These three identities altogether imply that $F(31^{n-3})$ is determined by $F(21^{n-2})$
and $F(1^n)$, and then the induction argument shows that $F$ is completely determined
by these two values.\qed 

\medskip Our discussion implies 
$$\hat{T}_{[det_n]}{\mathcal Dual}_{2n-2,n,n^2}=\hat{T}_{[det_n]}PGL(M_n).[det_n].$$
Theorem  \ref{smooth}   immediately follows.

\subsection{On the boundary of the orbit of the determinant}
Decompose a matrix $M$ into its symmetric and skew-symmetric parts 
$S$ and $A$. Define a polynomial $P_{\Lambda}\in S^n(M_n(\CC))^*$ by letting
$$
P_{\Lambda}(M)= \tdet_n(A,\ldots, A,S). 
$$
This is easily seen to be zero for $n$ even so we suppose $n$ to be odd. 
More explicitly, $P_{\Lambda}$ can   be expressed as follows. 
Let $Pf_i(A)$ denote the Pfaffian of the skew-symmetric matrix, of even size, 
obtained from $A$ by suppressing its $i$-th row and column. Then 
$$P_{\Lambda}(M)=\sum_{i,j}s_{ij}Pf_i(A)Pf_j(A).$$

\begin{prop}\label{Lambda}
The polynomial $P_{\Lambda}$ belongs to the orbit closure of the determinant. Moreover, 
$\ol{GL(W)\cdot P_{\Lambda}}$ is an irreducible codimension one component of the boundary of 
$\oldet$, not contained in $\tEnd(W)\cdot [\tdet_n]$. In particular $\ol{dc}(P_{\Lambda,m})=m<dc(P_{\Lambda,m})$.
\end{prop}

\proof The first assertion is clear: for $t\ne 0$, one can define an invertible 
endomorphism $u_t$ of $M_n(\CC)$ by $u_t(A+S)=A+tS$, where $A$ and $S$ are the skew-symmetric
and symmetric parts of a matrix $M$ in $M_n(\CC)$. Since the determinant of a skew-symmetric
matrix of odd size vanishes,  
$$(u_t\cdot \tdet_n)(M)=\tdet_n(A+tS)=nt\tdet_n(A,\ldots ,A,S)+O(t^2),$$
and therefore $u_t\cdot [\tdet_n]$ converges to $[P_{\Lambda}]$ when $t$ goes to zero. 

To prove the second assertion, we compute the stabilizer of $P_{\Lambda}$ inside 
$GL(M_n(\CC))$. The easiest way to make this computation uses the decomposition 
$\BC^n\ot \BC^n=\La 2\BC^n\oplus S^2 \BC^n$ of the space of matrices into skew-symmetric and 
symmetric ones. 
The action of $GL_n(\CC)$ on $M_n(\CC)$ by $M\mapsto gMg^t$ preserves $P_{\Lambda}$ up to
scale, and the Lie algebra of the stabilizer of $[P_{\Lambda}]$ is a $GL_n(\CC)$ 
submodule of $End(M_n(\CC))$. We have the decomposition into $GL_n(\CC)$-modules:
$$End(M_n(\CC))=End(\Lambda^2\op S^2)=End(\Lambda^2)\op End(S^2)
\op Hom(\Lambda^2,S^2)\op Hom(S^2,\Lambda^2).$$ 
Moreover, $End(\Lambda^2)={\mathfrak gl}_n\oplus EA$ and $End(S^2)={\mathfrak gl}_n\oplus ES$,
where $EA$ and $ES$ are distinct irreducible $GL_n(\CC)$-modules. Similarly, $Hom(\Lambda^2,S^2)=
{\mathfrak sl}_n\oplus EAS$ and $Hom(S^2,\Lambda^2)=
{\mathfrak sl}_n\oplus ESA$, where $EAS$ and $ESA$ are irreducible, pairwise distinct and 
different from $EA$ and $ES$. Then one can check that the modules $EA,ES,EAS,ESA$ are not contained
in the stabilizer, and that the contribution of the remaining terms is isomorphic 
with ${\mathfrak gl}_n\oplus {\mathfrak gl}_n$. In particular it has dimension $2n^2$,
which is one more than the dimension of the stabilizer of $[\det_n]$. This implies  
$\ol{GL(W)\cdot P_{\Lambda}}$ has codimension one in $\oldet$. Since it is not contained
in the orbit of the determinant, it must be an irreducible component of its
boundary. Since the zero set is not a cone (i.e., the equation involves all the variables), $P_{\Lambda}$ cannot
be in $\tEnd(W)\cdot \tdet_n$ which consists of $GL(W)\cdot \tdet_n$ plus cones. \qed

\medskip The hypersurface defined by $P_\Lambda$ has interesting properties. 

\begin{prop}\label{Lambdaex}
The dual variety of the hypersurface $Z(P_\Lambda)$ is isomorphic to 
the Zariski closure of  
$$\PP \{v^2\oplus v\wedge w\in S^2\BC^n\op \La 2\BC^n, \; v,w\in\CC^n\}
\subset \PP (M_n(\CC)).$$
\end{prop}

As expected, $Z(P_\Lambda)^*$ is close to being a Segre product $\PP^{n-1}\times \PP^{n-1}$. 
It can be defined as the image of the projective bundle 
$\pi : \PP(E)\rightarrow\PP^{n-1},$
where $E={\mathcal O}(-1)\oplus Q$ is the sum of the
tautological and quotient bundles on $\PP^{n-1}$, by a sub-linear system of 
${\mathcal O}_E(1)\otimes \pi^*{\mathcal O}(1)$. This sub-linear system
contracts the divisor $\PP(Q)\subset \PP(E)$ to the Grassmannian 
$G(2,n)\subset \PP \La 2\BC^n$.

\section{A large irreducible component of $ {\mathcal Dual}_{k,d,N}$}
Let $Sub_{k+2}(S^dW^*)$ be the projectivization of 
$$ \{  P \in S^dW^* \mid
\exists U^*\subset W^*,\   \tdim(U^*)=k+2, \ {\rm and}\ P\in S^dU^*\}, 
$$
  the {\it subspace variety} of hypersurfaces of degree $d$ in $\BP W$ that are cones
over some $Z^k\subset \BP^{k+1}\subset \BP W$.
The reduced, irreducible variety $Sub_{k+2}(S^dW^*)$   
  is of dimension $k+1+ (k+2)(N-(k+2))$ and its ideal
is generated in degree $k+3$ (see  \cite[\S 7.2]{weyman}).

If $[P]\in Sub_{k+2}(S^dW^*)$, then $Z(P)\subset \BP W$ is a cone with an $(N-k-1)$-dimensional vertex $\BP (U^*)\upperp$, and
$Z(P)^*\subset \BP U^*$. In particular $\tdim (Z(P)^*)\leq k$.

\begin{proposition} $Sub_{k+2}(S^d\BC^N)$ is a reduced, irreducible component of $ {\mathcal Dual}_{k,d,N}$.
\end{proposition}
\begin{proof}
Let $W^*=\BC^N$ and let  $P\in Sub_{k+2}(S^dW^*)$ be a general point. Write $P\in S^dU^*$. It follows from the Kempf-Weyman desingularization
described in \cite[\S 7.2]{weyman} that
$$
\hat T_{[P]}Sub_{k+2}(S^dW^*)= S^dU^* + \langle (u\intprod P)\circ \a \mid u\in U, \ \a\in W^*\rangle .
$$
If we choose a complement $V^*$ to $U^*$ in $W^*$ we may write
$$
\hat T_{[P]}Sub_{k+2}(S^dW^*)= S^dU^* \op  \langle (u\intprod P)\circ \a \mid u\in U, \ \a\in V^*\rangle.
$$
We must show
$$\hat T_{[P]} {\mathcal Dual}_{k,d,N}\subseteq \hat T_{[P]}Sub_{k+2}(S^dW^*)
\subset S^dU^*\op S^{d-1}U^*\ot V^* \subset S^d(U^*\op V^*).
$$

Following the notation and discussion of \S\ref{sec32},
  in equation \eqref{Zarb}, for $P\in Sub_{k+2}(S^dW^*)$, since the determinant of the Hessian on any
$k+3$-plane vanishes,  $Q_F=0$, and  
we conclude
$
H_{\pi,w}(X)=0
$
for all $[w]\in Z(P)$ and for {\it all} $X\in V$.
This says the degree $d-2$ hypersurface $Z(H_{\pi,-}(X))$
is contained in the {\it irreducible} degree $d$ hypersurface $Z(P)$, which implies
$\pi(X,X,\cdot \hd \cdot)= \frac{\partial^2\pi}{(\partial X)^2}=0$ for all $X\in V$, i.e.,
$\hat T_{[P]} {\mathcal Dual}_{k,d,N}\subseteq S^dU^*\op S^{d-1}U^*\ot V^*$.
To obtain the restrictions on the term in $S^{d-1}U^*\ot V^*$ we must
consider the term of order two in $\ep$
in the expansion of  $\tdet(H_{P_{\ep}}\mid_F)=P_{\ep}Q_{\ep,F}$.   With our choice of splitting we may identify $U=(V^*)\upperp \subset W$ and   take $F'=U$.
(In other words, the choice of $F'$ is equivalent to choosing the splitting.)
Note the order $\ep$ term also implies in this case that 
$
Q_F'=0$.

 The terms on the left hand side that potentially could contribute to the $\ep^2$ coefficient are in
$$ 
\det \begin{pmatrix}
\frac{\partial^2P}{\partial u_i\partial u_j} + \ep \frac{\partial^2\pi}{\partial u_i\partial u_j} & \ep \frac{\partial^2\pi}{\partial u_i\partial X}\\
\ep \frac{\partial^2\pi}{\partial u_j\partial X} & \ep \frac{\partial^2\pi}{( \partial X)^2} + \ep^2 \frac{\partial^2\t}{( \partial X)^2}
\end{pmatrix}
$$
The actual contribution is the sum of $\tdet_{k+2}(H_P|_U)\frac{\partial^2\t}{( \partial X)^2}$
and terms substituting  two entries from
$ \frac{\partial^2\pi}{\partial u_i\partial X}$ for two of the columns of
$H_P|_U=\frac{\partial^2P}{\partial u_i\partial u_j}$. 
The right hand side is zero.

Choose $[w]\in Z(P)\cap Z(\tdet_{k+2}(H_P|_U))$, and note that we can take a basis of
elements of $W$ of this form, so the first term is zero.
We conclude that the column vector $ \frac{\partial^2\pi}{\partial u_i\partial X}$ is a linear
combination of the columns of
$\frac{\partial^2P}{\partial u_i\partial u_j}$ which implies
$ \frac{\partial \pi}{ \partial X}$ is a linear combination of the 
$\frac{\partial P}{\partial u_i }$, i.e. that $ \frac{\partial \pi}{ \partial X}=u\intprod P$ for some $u\in U$ which is what we needed to prove.
\end{proof}

\bibliographystyle{amsplain}
 
\bibliography{Lmatrix}

\end{document}